\documentclass{amsart}

\usepackage[english]{babel}

\usepackage[letterpaper,margin=1in]{geometry}

\usepackage{enumitem}

\usepackage{amsmath, amssymb, amsthm}
\usepackage{nicematrix}
\usepackage{tikz}
\NiceMatrixOptions{code-for-first-row = \color{red},
code-for-first-col = \color{blue}
}

\newcommand{\Hess}{{\operatorname{Hess}}}
\newcommand{\Jac}{{\operatornamewithlimits{Jac}}}

\newcommand{\Vectorization}{{\operatorname{Vec}}}

\usepackage{graphicx}
\usepackage[colorlinks=true, allcolors=blue]{hyperref}
\usepackage{xcolor}

\newtheorem{theorem}{Theorem}[section]
\newtheorem{lemma}[theorem]{Lemma}
\newtheorem{proposition}[theorem]{Proposition}

\newtheorem{conjecture}[theorem]{Conjecture}

\theoremstyle{definition}
\newtheorem{definition}[theorem]{Definition}
\newtheorem{example}[theorem]{Example}
\newtheorem{question}[theorem]{Question}
\newtheorem{notation}[theorem]{Notations}

\theoremstyle{remark}
\newtheorem*{remark}{Remark}

\title{On the Structure of Second Jacobian Ideals}

\author{Fei Ye}
\address{Department of Mathematics and Computer Science, Queensborough Community College of CUNY, 222-05, 56th Avenue Bayside, NY 11364}
\address{
Department of Mathematics, CUNY Graduate Center, 365 Fifth Avenue, New York, NY 10016
}
\email{feye@qcc.cuny.edu}



\thanks{
The author is partially supported by PSC-CUNY Award \# 67281-00 55. The author sincerely appreciates the referee's invaluable comments and suggestions, which have significantly enhanced the clarity and rigor of the paper.
}

\begin{document}

\begin{abstract}
We show that the second Jacobian ideal of a hypersurface can be decomposed such that a power of the Jacobian ideal becomes a factor.
As an application of the decomposition, we present an elementary proof establishing that the second Nash blow-up algebra of a hypersurface singularity is a contact invariant.
\end{abstract}

\maketitle

\section{Introduction}

Let $X$ be a complex variety of dimension $d$, $x \in X$ a point, and $x^{(k)} := \operatorname{Spec}\left(\mathcal{O}_{X, x}/\mathfrak{m}_x^{k+1}\right)$ its $k$-th infinitesimal neighborhood. For any point $x$ in the smooth locus $X_{\mathrm{sm}}$ of $X$, the Artinian subscheme $x^{(k)}$ of $X$ corresponds to a point $[x^{(k)}]$ in the Hilbert scheme $\operatorname{Hilb}_{\binom{d+k}{k}} (X)$ of $\binom{d+k}{k}$ points of $X$. 
T. Yasuda \cite{Yasuda2007} shows that the map 
$$\sigma_k: X_{\mathrm{sm}} \to \operatorname{Hilb}_{\binom{d+k}{k}} (X), \quad x\mapsto [x^{(k)}]$$ is a morphism of schemes and defines the $k$-th Nash blow-up of $X$ as $\pi_k: \mathrm{Nash}_k(X)\to X$, where $\mathrm{Nash}_k(X)$ is the closure of the graph $\Gamma_{\sigma_k} \subset X_{\mathrm{sm}} \times_\mathbb{C} \operatorname{\bf Hilb}_{\binom{d+k}{k}} (X)$ and $\pi_k$ is the first projection restricted to $\mathrm{Nash}_k(X)$.


The $k$-th Nash blow-ups and their properties in various contexts have been investigated in the recent literature, such as \cite{Yasuda2007}, \cite{Barajas2020}, and \cite{Le2024}.
A key finding is that the $k$-th Nash blow-up corresponds to the blow-up of the $k$-th Jacobian ideal. 
For a hypersurface, D. Duarte \cite{Duarte2017} explicitly characterizes the $k$-th Jacobian ideal as the ideal generated by the maximal minors of the $k$-th Jacobian matrix. 
It is worth mentioning that higher Jacobian matrices can also be defined for an ideal of holomorphic functions (see \cite{Barajas2020}). 
Higher Jacobian matrices have played important roles in the studies of modules of K\"ahler differentials of higher orders (\cite{Barajas2020}), arc schemes \cite{Dreau2024}, and motivic zeta functions \cite{Le2024}.

In this paper, we focus on the second Jacobian matrix of a hypersurface and show that the second Jacobian ideal can be decomposed into the sum of two ideals: a power of the Jacobian ideal and an ideal involving second-order partial derivatives of $F$.

The motivation of this project stems from a conjecture by N. Hussain, G. Ma, S. S.-T. Yau, and H. Zuo (\cite[Conjecture 1.5]{Hussain2023}). The conjecture was confirmed for $n=2$ and $k=2$ in \cite{Hussain2023}. In \cite{Ramirez2024}, we confirmed the conjecture for $n=3$ and $k=2$ using an explicit description of the second Jacobian ideal. After the completion of \cite{Ramirez2024}, we learned that \cite[Conjecture 1.5]{Hussain2023} has been confirmed in general by Q. T. Lê and T. Yasuda in \cite{Le2024} using Fitting ideals. 

In \cite{Ramirez2024}, we proposed a conjecture regarding the structure of the second Jacobian ideal and confirmed it for $n = 2$ and $n = 3$.

\begin{conjecture}[{\cite[Conjecture 2.8.]{Ramirez2024}}]\label{conj:decomposition}
    Let $F$ be a holomorphic function of $n$ variables and $K$ the $(n+1)$-by-$\binom{n+1}{2}$ submatrix of the second Jacobian $\Jac_2(F)$ consisting of the columns $\beta_{ab}$ with $0< a\le b \le n$. The ideal of maximal minors of $K$ is 
\[
\mathcal{J}_1(F)^{n-2}\mathcal{Q}(F),
\]
where $\mathcal{J}_1(F)$ is the Jacobian ideal of $F$,
$$\mathcal{Q}(F)= \big(Q_{ij;kl}(F)\big)_{1\le i, j, k,l\le n},$$ and 
$$Q_{ij;kl}(F)=\partial_{ik}F\partial_jF\partial_lF - \partial_{jk}F\partial_iF\partial_lF - \partial_{il}F\partial_jF\partial_kF + \partial_{jl}F\partial_iF\partial_kF.$$ 
\end{conjecture}

In this paper, we show that this conjecture holds true for any $n > 3$. 

The paper is organized as follows. 
In Section 2, we recall the definition of the second Jacobian matrix of a hypersurface and examine its structure. 
In Section 3, we present results on the minors of the second Jacobian matrix, which facilitates reductive calculations of these maximal minors. 
In Section 4, we show that Conjecture \ref{conj:decomposition} holds true.
In Section 5, we provide an elementary proof of \cite[Conjecture]{Hussain2023} for $k = 2$.

\section{second Jacobian matrices}\label{sec:higherJacobian}

Let $\mathbb{N}_{0}$ be the set of nonnegative integers and $\mathbb{N}_0^m$ the set of $m$-tuples in $\mathbb{N}_0$.

Define $\beta_i=(0, \ldots, 0, \overset{i\text{th}}{1},0\ldots, 0)$ and let $\beta_0$ be the zero vector. 
For a vector $\alpha=(a_1,\ldots, a_m)\in \mathbb{N}_{0}^m$, we denote $\alpha!=a_1!\cdots a_m!$, $\partial_i=\frac{\partial}{\partial x_i}$, and $\partial^\alpha=\partial_{i_1}^{a_{i_1}}\cdots \partial_{i_r}^{a_{i_r}}$, where $a_{i_1}$, $\cdots$, $a_{i_r}$ are the nonzero entries of $\alpha$.

Let $F$ be a holomorphic function of $n$ variables.  The second Jacobian matrix of $F$, denoted as $\Jac_2(F)$, is an extension of the Jacobian matrix. It consists of $n+1$ rows labeled as $\beta_0$, $\beta_1$, $\cdots$ $\beta_n$, and $n+\binom{n+1}{2}$ columns labeled as $\beta_{01}$, $\cdots, \beta_{0n}$, $\beta_{11}$, $\cdots$, $\beta_{1n}$, $\beta_{22}$, $\cdots$, $\beta_{(n-1)n}$, $\beta_{nn}$, where $\beta_{ij} :=\beta_i + \beta_j$. Both rows and columns are ordered in graded lexicographical order as listed above.

Given a matrix $M$, we denote the $(i, j)$ entry of $M$ as $M(i,j)$.

An index of a row (column) label is simply called a row (column) index.

\begin{definition}
    Let $F$ be a holomorphic function of $n$ variables.  The \textbf{second Jacobian matrix} $\Jac_2(F)$ of $F$ is the $(n+1)\times \left(n+\binom{n+1}{2}\right)$ matrix whose $(\beta_k, \beta_{ij})$ entry, where $i\le j$, is 
    \[
    {\Jac}_2(F)(\beta_k, \beta_{ij})
    =
    \begin{cases}
         \frac{1}{2}\partial_{ii}(F) & \text{ if } k=0 \text{ and } i=j\\
         \partial_{ij}(F) & \text{ if } k=0 \text{ but } i\neq j\\
         \partial_j(F) & \text{ if } k=i\ge 1,\\
         \partial_i(F) & \text{ if } k=j\ge 1, \text{ and } i>0\\
         0 & \text{otherwise}.
    \end{cases}
    \]

    The \textbf{second Jacobian ideal} $\mathcal{J}_2(F)$ of $F$ is the ideal generated by maximal minors of $\Jac_2(F)$.

    The \textbf{(first) Jacobian ideal} $\mathcal{J}_1(F)$ of $F$ is the ideal generated by $\partial_iF$, where $1\le i\le n$.
\end{definition}

\begin{remark}
    The above definition is adopted from \cite{Le2024} and \cite{Barajas2020}, which differs to the definition in \cite{Duarte2017} only in the $(\beta_{k}, \beta_{0k})$ entries.
    In \cite{Duarte2017}, the $(\beta_{k}, \beta_{0k})$ entries are $F$ instead of $0$. 
    It is clear that these two versions coincide modulo $F$.
\end{remark}

\begin{notation} For simplicity, we fix some notations. Let $F$ be a holomorphic function of $n$ variables $x_1, \dots, x_n$.

    We will denote $\partial_k F$ by $f_k$ and $\partial_{ij}F=\partial_i(\partial_jF)$ by $f_{ij}$ for any $1\le i,j, k\le n$.

    For $1\le i, j, k, l \le n$, we let
    \[Q_{ij;kl}(F) =  f_{ik}f_jf_l - f_{jk}f_if_l - f_{il}f_jf_k + f_{jl}f_if_k,\]
    and 
    \[\mathcal{Q}(F) = (Q_{ij;kl}(F))\]
    the ideal generated by $Q_{ij;kl}(F)$, where $1\le i,j,k,l\le n$.

    We use the hat notation $\widehat{\ast}$ to omit the label $\ast$ together with the associated delimiter or operator.
    
    For a matrix $M$, we denote by $M([a_1, \cdots, a_r], [b_1, \cdots, b_s])$ the submatrix consisting of the $(a_k,b_l)$ entries of $M$, where $a_k$ and $b_l$ are row labels and column labels of $M$ respectively. 
    If the set $\{a_1, \dots, a_r\}$ consists of all row labels of $M$, then we simply denote the submatrix by $M(: , [b_1, \cdots, b_s])$. Similarly, we denote by $M([a_1, \cdots, a_r], :)$ the submatrix of $M$ consisting of the rows $a_k$ and all columns.

    Regarding column labels, we always identify $\beta_{ji}$ with $\beta_{ij}$ if $j>i$.

    For a submatrix $D$ of $\Jac_2(F)$, given a row index $i$ of $D$, we denote by $m_i(D)$ is the frequency of $i$ appearing in the column labels of $D$. We will write $m_i(D)$ as $m_i$ when the context makes it clear and there is no risk of confusion.

    We denote that two matrices $A$ and $B$ are equivalent up to permutations on rows and columns by $A\sim_\sigma B$. For simplicity, we say $A$ and $B$ are \textbf{permutation equivalent}.
\end{notation}

\begin{remark}
    Note that $f_{ij} = f_{ji}$, which implies the following equalities
    \[Q_{ij;kl}(F) = 0 \text{ if } i=j \text{ or } k=l,\]
    \[Q_{ij;kl}(F)=-Q_{ji;kl}(F)=-Q_{ij;lk}(F)=Q_{ji;lk}(F)=Q_{kl;ij}(F).\]
    Therefore, the ideal $\mathcal{Q}(F)$ is generated by the set $\{Q_{ij;kl}(F) \mid 1\le i < j \le n, 1\le k < l\le n\}$.
\end{remark}

\begin{example}
    For $n=3$, the second Jacobian matrix $\Jac_2(F)$ is a $4\times 9$ matrix of holomorphic functions.
\[
\begin{pNiceMatrix}[first-row, first-col]
& \beta_{01} & \beta_{02} & \beta_{03} 
            & \beta_{11} & \beta_{12} & \beta_{13} 
                              & \beta_{22} & \beta_{23} 
                                                & \beta_{33}\\
    \beta_0 & f_1   & f_2          & f_3 
            & \dfrac{1}{2}f_{11}  & f_{12}  &  f_{13}
                                    & \dfrac{1}{2}f_{22}  &  f_{23}
                                                                   &  \dfrac{1}{2}f_{33}\\
    \beta_1 & 0   &               0              &  0 
            & f_{1}  & f_{2}  &  f_{3}
                                    & 0  &  0
                                                                   &  0\\
    \beta_2 & 0   &               0              &  0 
            & 0  & f_{1}  &  0
                                    & f_{2}  &  f_{3}
                                                                   &  0\\
    \beta_3 & 0   &               0             &  0 
            & 0  & 0  &  f_{1}
                                    & 0  &  f_{2}
                                                                   &  f_{3}\\
\end{pNiceMatrix}.
\]
\end{example}

Due to the nature of $\Jac_2(F)$, each column contains at most $3$ nonzero entries. It follows that square submatrices of $\Jac_2(F)$ have a nice structure, making it easier to compute their determinants.

Our first observation is that square submatrices are permutation equivalent to block upper triangular matrices.

\begin{example}
Let $F$ be a polynomial in $\mathbb{C}[z_1, \cdots, z_8]$. Consider the square submatrix 
\[
M = \Jac_2(F)\left(:, [\beta_{11}, \beta_{12}, \beta_{13}, \beta_{14}, \beta_{15}, \beta_{38}, \beta_{46}, \beta_{66}, \beta_{77}]\right).
\]

Since the indices $2$, $5$ and $8$ appear only once, the submatrix $B_1 = M([\beta_2, \beta_5, \beta_8], [\beta_{12}, \beta_{15}, \beta_{38}])$ is a diagonal matrix.

Denote the complement matrix of $B_1$ by $C_1$ which is obtained from $M$ by removing the rows and the columns of $M$ associated with $B_1$. Rearranging the rows and columns of $M$ so that $B_1$ is in the lower right corner leads to
\[
M \sim_\sigma\begin{pmatrix}
    C_1 & \ast \\
    0 & B_1
\end{pmatrix},
\]
and
\[
C_1 = M([\beta_0,\beta_1,\beta_3, \beta_4,\beta_6,\beta_7], [\beta_{11}, \beta_{13}, \beta_{14}, \beta_{46}, \beta_{66}, \beta_{77}] ).
\]

Note that $3$ appears only once as a column index of $C_1$. Let $B_2=C_1([\beta_{3}], [\beta_{13}])$ and 
\[
C_2=C_1([\beta_0,\beta_1, \beta_4,\beta_6,\beta_7], [\beta_{11}, \beta_{14}, \beta_{46}, \beta_{66}, \beta_{77}]).
\]
Then 
\[
C_1 \sim_\sigma\begin{pmatrix}
    C_2 & \ast \\
    0 & B_2
\end{pmatrix}.
\]
Because $7$ appears twice but only in one column label in $C_2$, we take $B_3 = C_2([\beta_7], [\beta_{77}])$ and $$D=C_2[\beta_0,\beta_1, \beta_4,\beta_6], [\beta_{11}, \beta_{14}, \beta_{46}, \beta_{66}]).$$
Rearranging the rows and the columns of $M$ so that $B_1$, $B_2$, $B_3$, and $D$ are in the diagonal from bottom to top leads to
\[
M \sim_\sigma\begin{pmatrix}
   D & \ast & \ast & \ast \\ 
   0 & B_3 & \ast & \ast \\
   0 & 0 & B_2 & \ast \\
   0 & 0 & 0 & B_1
\end{pmatrix} = \begin{pNiceMatrix}[first-row, first-col, margin, cell-space-limits=2pt, code-for-first-col = \color{blue}]
\CodeBefore
\rectanglecolor{blue!10}{7-7}{9-9}
\rectanglecolor{blue!30}{6-6}{6-6}
\rectanglecolor{blue!50}{5-5}{5-5}
\rectanglecolor{blue!60}{1-1}{4-4}
\Body
   &\beta_{11} & \beta_{14} & \beta_{46}  & \beta_{66}   & \beta_{77}  & \beta_{13}   & \beta_{12} &  \beta_{15} & \beta_{38}   \\
   \beta_0 & \frac{1}{2}\partial_{11}F & \partial_{14}F & \partial_{46}F & \frac{1}{2}\partial_{66}F & \frac{1}{2}\partial_{77}F & \partial_{13}F   & \partial_{12}F & \partial_{15}F & \partial_{38}F \\
   \beta_{1} & \partial_1F & \partial_4F & 0 & 0 & 0 & \partial_3F & \partial_2F & \partial_5F & 0  \\
   \beta_{4} & 0  & \partial_1F & \partial_6F & 0 & 0 & 0 & 0 & 0 & 0 \\
   \beta_{6} & 0  & 0 & \partial_4F & \partial_6F & 0 & 0 & 0 & 0 &  0\\
   \beta_{7} & 0 & 0 & 0 & 0 & \partial_7F &  0 & 0 & 0 & 0\\
   \beta_{3} & 0 & 0 & 0 & 0 & 0 & \partial_1F & 0 & 0 & \partial_8F \\
   \beta_{2} & 0 & 0 & 0 & 0 & 0 & 0 & \partial_1F & 0 & 0 \\
   \beta_{5} & 0 & 0 & 0 & 0 & 0 & 0 &0& \partial_1F & 0\\
   \beta_{8} & 0 & 0 & 0 & 0 & 0 & 0 & 0 & 0 & \partial_3F \\
\end{pNiceMatrix}.
\]
Notice that each nonzero row index of $D$ appears in at least two column labels and each column index of $D$ is also a row index of $D$. For example, the row index $1$ appears in the column labels $\beta_{11}$ and $\beta_{14}$ and the column index $4$ is also a row index.
\end{example}

By generalizing the construction from the example, we obtain the following result, which will be used inductively in later proofs.

\begin{lemma}\label{lem:blockUpperDiagonal}
Let $S\subseteq \{1, \dots, n\}$ be a subset of cardinality $r$ and $M$ an $(r+1)$-by-$(r+1)$ submatrix of $\Jac_2(F)$, where the set of column indices of $M$ is $S$ and the set of row indices of $M$ is $\{0\}\cup S$. Then the matrix $M$ is permutation equivalent to a block upper triangular matrix
\[
M \sim_\sigma
\begin{pmatrix}
   D & \ast & \ast & \ast \\ 
   0 & B_m & \ast & \ast \\
   0 & 0 & \ddots & \ast \\
   0 & 0 & 0 & B_1
\end{pmatrix}
\]
where $B_i$ are diagonal square matrices (could be $0$-dimensional) and $D$ is a square matrix obtained by removing the rows and the columns associated with $B_i$ for $i=1, \dots, m$. 
   Moreover, each nonzero row index of $D$ appears in at least two column labels of $D$ and each column index of $D$ is also a row index of $D$.  
\end{lemma}

\begin{proof}

We will construct a block upper triangular matrix permutation equivalent to $M$ through a reductive process.

\paragraph{\textbf{Step 1.}}
Suppose that $S$ has a partition $S=U\cup V$, where 
each index in $U$ appears only once in the column labels of $M$ and each index in $V$ appears at least twice in column labels of $M$. 

\paragraph{\textbf{Case 1.1}}
Suppose that $U$ is empty. Then $M$ is a square matrix in which each nonzero row index appears at least twice in column labels of $M$. Therefore, we let $C=M$. 

\paragraph{\textbf{Case 1.2}}
Suppose $U$ is nonempty. Let $U_1 = U = \{u_1, \dots, u_a\}$ and
\[
B_1 = M([\beta_{u_1}, \dots, \beta_{u_a}],[\beta_{u_1v_1}, \dots, \beta_{u_av_a}]),
\]
where $u_i\in U_1$ and $u_1<u_2<\cdots<u_a$. Let $V_1=\{v_1, \dots, v_a\}$. Since $u_i\in U$ appears only once in the column labels of $M$, it follows that $U_1\cap V_1 =\emptyset$ and $V_1\subset V$. Then 
\[B_1(\beta_{u_i}, \beta_{u_jv_j})=  f_{v_j} \text{ if } i=j, ~\text{or}~ B_1(\beta_{u_i}, \beta_{u_jv_j})= 0 \text{ otherwise}.
\]
It follows that $B_1$ is a diagonal matrix whose diagonal entries are $f_{v}$, $v\in V_1$.
Denote by $C$ the complement matrix obtained from $M$ by removing the rows and columns associated with $B_1$. The set of row indices of $C$ is $(\{0\}\cup S)\setminus U_1 = \{0\}\cup V$ and the set of column indices of $C$ is $S\setminus (U_1 \cup V_1) \subset V$.
Therefore, for any $u_i\in U_1$, the entry $M(\beta_{u_i}, \beta_{x y}) = 0$ if $\beta_{x y}$ is a column label of $C$.
Rearranging the rows and columns of $M$ so that $B_1$ is in the lower right corner leads to
\[
M \sim_\sigma\begin{pmatrix}
    C & \ast \\
    0 & B_1
\end{pmatrix}.
\]


Since $S$ is finite, repeating the aforementioned argument for $C$ and always denoting the top diagonal block as $C$ results in the following 
\[
M \sim_\sigma\begin{pmatrix}
    C & \ast & \ast & \ast\\
    0 & B_{m-1} & \ast & \ast\\
    0 & 0 & \ddots & \ast\\
    0 & 0 & 0 & B_1 \\
\end{pmatrix},
\]
where $B_i$ are diagonal matrices and $C$ is a square submatrix.

From the construction, we see that each nonzero row index of $C$ appears at least twice in  column labels of $C$ and each column index of $C$ is also a row index of $C$.

\paragraph{\textbf{Step 2.}}
We still denote by $S$ the set of nonzero row indices of $C$. Suppose $S$ has a partition $S = W \cup V$ such that each index in $W$ appears twice but only in one column label and each index in $V$ appears in at least two column labels.

\paragraph{\textbf{Case 2.1}}
Suppose that $W$ is empty. Then $C$ is a square submatrix in which each nonzero row index appears in at least two column labels. We let $D = C$ and obtain a block upper triangular matrix permutation equivalent to $M$.

\paragraph{\textbf{Case 2.2}}
Suppose $W$ is nonempty. Let 
\[
B_m =\begin{pmatrix}
    M(\beta_{w}, \beta_{ww})
\end{pmatrix}_{w\in W}.
\]
Then $B_m$ is a diagonal matrix whose diagonals are $f_w$ with $w\in W$. Let $D$ be the complement matrix obtained from $C$ by removing the rows and the columns associated with all $B_m$. Since $u_i\neq w_j$ for any $u_i \in S\setminus W$ and $w_j\in W$, $C(\beta_{w_i},\beta_{xy}) = 0$ for any column label $\beta_{xy}$ of $D$. Therefore, rearranging the rows and the columns of $M$ can transform $C$ into a block upper diagonal matrix with $D$ and $B_m$ on the diagonal. 
Together with the construction in Step 1, this leads to
\[
M \sim_\sigma\begin{pmatrix}
    D & \ast & \ast & \ast\\
    0 & B_m & \ast & \ast\\
    0 & 0 & \ddots & \ast\\
    0 & 0 & 0 & B_1 \\
\end{pmatrix},
\]
where $B_i$ are diagonal matrices and $D$ is a square submatrix. Moreover, from the construction, it is clear that each nonzero row index of $D$ appears in at least two column labels of $D$ and each column index of $D$ is also a row index of $D$.
\end{proof}

We end this section with an observation on column labels of the submatrix $D$ appearing in Lemma \ref{lem:blockUpperDiagonal}.

\begin{lemma}\label{lem:doubleindices}
    Let $D$ be a $(r+1)$-by-$(r+1)$ submatrix of $\Jac_2(F)$ that contains the row $\beta_0$. 
    Assume that each nonzero row index of $D$ appears in at least two column labels of $D$ and that each column index of $D$ is also a nonzero row index of $D$.
    Then, either there is one row index appearing four times as a column index, or there are two row indices that each appears three times as column indices. In either case, each of the remaining nonzero row indices appears exactly twice as a column index.
    Moreover, $r\ge 2$.
\end{lemma}

\begin{proof}
Let $S$ be the set of nonzero row indices of $D$. For each $i\in S$, recall that $m_i$ is the frequency of $i$ appearing in the column labels of $D$. Since the dimension of $D$ is $r+1$, the cardinality of $S$ is $r$. Because there are $r+1$ column labels and each has two indices, we have
\[\sum_{a\in S}m_a = 2(r+1).\]
Since $m_a \ge 2$ for all $a\in S$, for any $i, j \in S$ with $i\neq j$, we have
\[\sum_{a \in S, a \neq i} m_a \ge 2(r - 1) \quad \text{and} \quad \sum_{a \in S, a \neq i, j} m_a \ge 2(r - 2).\]
Moreover, the equalities hold only if the summands $m_a=2$.
Then, for each $i\in S$, we have
\[
2\le m_i = \sum_{a\in S}m_a - \sum_{a \in S, a \ne i}m_a \le 4.\]
Because the cardinality of $S$ is $r$, there must be at least one $i$ such that $m_i>2$. Otherwise, there would be a contradiction $2(r+1) = \sum\limits_{a\in S}m_a = 2r $.

If there is an index $i\in S$ such that $m_i = 4$, then 
\[\sum_{a\in S, a\ne i} m_a = 
2(r-1).\]
For any $j\in S\setminus\{i\}$, we then have
\[2\le m_j = \sum_{a \in S, a \ne i} m_a - \sum_{a\in S, a\ne i, j} m_a \le 2\]
which implies that $m_j=2$ for $j\in S\setminus\{i\}$.
Since $m_i=4$ and a column label consists of two indices, the index $i$ must appear in at least three different column labels, which implies $r \ge 2$.

If there is an index $i\in S$ such that $m_i=3$, then
\[2(r-1)\le \sum_{a \in S, a \ne i}m_a = 
2r - 1.\]
Therefore, there must be another index $j\in S\setminus \{i\}$ such that $3\le m_j \le 4$. By Case 1, $m_j\ne 4$. Then it follows that $m_i=m_j=3$. Similar to Case 1, $m_a=2$ for any $a \in S\setminus\{i,j\}$. Furthermore, the indices $i$ and $j$ must appear in at least three different column labels, which implies $r \ge 2$.

\end{proof}

\section{Some Results on Minors}

In this section, we present some results on minors that will help us to calculate maximal minors.

\begin{lemma}[{\cite[Lemma 2.3]{Ramirez2024}}]\label{lem:detSubmatrices}
Let $F$ be a holomorphic function of $n$ variables and $S$ a square submatrix of $\Jac_2(F)$ consisting of the rows $\beta_{i_1}$, $\dots$, $\beta_{i_r}$ and the columns $\beta_{i_1j_1}$, $\dots$, $\beta_{i_rj_r}$ with $1\le i_1< \cdots <i_r \le n$. If $1\le j_1\le j_2\le\cdots\le j_r \le n$, then
\[\det(S)= \prod_{a=1}^r f_{j_a}.\]
\end{lemma}

\begin{lemma}\label{lem:degenerateSubmatrix}
    Let $M$ be a $(n+1)$-by-$(n+1)$ submatrix of $\Jac_2(F)$ consisting of the columns $\beta_{i_1j_1}$, $\dots$, $\beta_{i_{n+1}j_{n+1}}$, where $U:=\bigcup\limits_{a=1}^{n+1}\left(\{i_a\}\cup\{j_a\}\right)\subseteq\{1, \dots, n\}.$    
    If $|U| < n$, then $\det(M)=0$.
\end{lemma}
\begin{proof}
Suppose that $i\in \{1, \dots, n\}\setminus U$. Then the row $\beta_i$ is a zero vector. Therefore, $\det(M)=0$.
\end{proof}

Consider the submatrix
\[R=\Jac_2(F)\left([\beta_0, \beta_i, \beta_j, \beta_k, \beta_l, \beta_{u_1},\dots, \beta_{u_{r-4}}], [\beta_{11}, \dots, \beta_{1n}, \beta_{22}, \dots, \beta_{(n-1)n}, \beta_{nn}]\right),\]
where the row labels are all different.
Let $S=\{i,j,k,l,u_1, \cdots, u_{r-4}\}$ be the set of nonzero row indices of $R$. 

\begin{lemma}\label{lem:ExpansionMaximalMinorA}
Consider the $(r+1)$-by-$(r+1)$ matrices
\[M=R(:, [\beta_{ik}, \beta_{jk}, \beta_{jl}, \beta_{i_1j_1}, \dots, \beta_{i_{r-2}j_{r-2}}])\]
    and
    \[N=R(:, [\beta_{il}, \beta_{jj}, \beta_{kk}, \beta_{i_1j_1}, \dots, \beta_{i_{r-2}j_{r-2}}]),\]
    where $U:=\bigcup\limits_{a=1}^{r-2}\left(\{i_a\}\cup\{j_a\}\right) \subseteq S\setminus\{j,k\}$.
    Then
    \[\det(M) + \det(N) = -Q_{ij;kl}(F)\det(A),\]
    where $A$ is the submatrix obtained from $M$ by removing the rows $\beta_0$, $\beta_j$, $\beta_k$ and the columns $\beta_{ik}$, $\beta_{jk}$, $\beta_{jl}$. 
\end{lemma}
\begin{proof}
Note that $j$, $k$ are not in $U$. The $(\beta_j,\beta_{i_aj_a})$ and $(\beta_k, \beta_{i_aj_a})$ entries of both $M$ and $N$ are zero for all $a \in \{1, \dots, r-2\}$. Then, any $(r-2)$-by-$(r-2)$ submatrix of $M(:, [ \beta_{i_1j_1}, \dots, \beta_{i_{r-2}j_{r-2}}]) = N(:, [\beta_{i_1j_1}, \dots, \beta_{i_{r-2}j_{r-2}}])$ that contains either the row $\beta_j$ or the row $\beta_k$ has a determinant zero. Therefore, applying Laplace's cofactor expansion along the first three columns $\beta_{ik}$, $\beta_{jk}$ and $\beta_{jl}$ of $M$ leads to
    \[
    \begin{split}
        \det(M) = & \det(A)\begin{vmatrix}
        f_{ik} & f_{jk} & f_{jl}\\
        0    & f_k    & f_l \\
        f_i  & f_j    & 0
    \end{vmatrix} 
    - \det(B) \begin{vmatrix}
        f_k & 0   & 0\\
        0 & f_k & f_l \\
        f_i   & f_j & 0
    \end{vmatrix} 
    + \det(C) \begin{vmatrix}
        0   & f_k & f_l \\
        f_i & f_j & 0 \\
        0   & 0   & f_j
    \end{vmatrix} \\
    = & \det(A)\left(-f_{ik}f_jf_l + f_{jk}f_if_l - f_{jl}f_if_k\right) + \det(B)f_jf_kf_l - \det(C)f_if_jf_k.
    \end{split}
    \]
Similarly, applying Laplace's cofactor expansion along the first three columns $\beta_{il}$, $\beta_{jj}$ and $\beta_{kk}$ of $N$ leads to 
    \[
    \begin{split}
        \det(N) = & \det(A)\begin{vmatrix}
        f_{il} & \frac12 f_{jj} & \frac12 f_{kk}\\
        0 & f_j    & 0 \\
        0     & 0  & f_k 
    \end{vmatrix} 
    - \det(B) \begin{vmatrix}
       f_l      & 0  & 0\\
       0 &  f_j    & 0   \\
       0 & 0      & f_k 
    \end{vmatrix} 
    + \det(C) \begin{vmatrix}
        0  & f_j    & 0  \\
        0 & 0      & f_k  \\
        f_i      & 0     & 0
    \end{vmatrix} \\
    = & \det(A)f_{il}f_jf_k - \det(B)f_jf_kf_l + \det(C)f_if_jf_k.
    \end{split}
    \]

Taking the sum yields
\[\det(M) + \det(N) = -Q_{ij;kl}(F)\det(A).\]
\end{proof}

\begin{lemma}\label{lem:ExpansionMaximalMinorB}
Consider the $(r+1)$-by-$(r+1)$ matrices
\[M=R(:, [\beta_{ii}, \beta_{ij}, \beta_{jk}, \beta_{i_1j_1}, \dots, \beta_{i_{r-2}j_{r-2}}])\]
    and
    \[N=R(:, [\beta_{ii}, \beta_{ik}, \beta_{jj}, \beta_{i_1j_1}, \dots, \beta_{i_{r-2}j_{r-2}}]),\]
    where $U:=\bigcup\limits_{a=1}^{r-2}\left(\{i_a\}\cup\{j_a\}\right) \subseteq S$.
\begin{enumerate}
    \item \label{lem-item:mi=4}
    If $j$ is not in $U$, then
    \[\det(M) + \det(N) =  Q_{ij;ik}(F) \det(A),\]
    where $A$ is the $(r-2)$-by-$(r-2)$ submatrix obtained from $M$ by removing the rows $\beta_0$, $\beta_i$ $\beta_j$ and the columns $\beta_{ii}$, $\beta_{ij}$, $\beta_{jk}$. 

    \item \label{lem-item:mi=3}
    If $i$ is not in $U$, then
    \[\det(M) + \det(N) = Q_{ij;ik}(F) \det(A)  - \frac12 Q_{ij;ij}(F) \det(B),\]
    where $A$ is the submatrix obtained from $M$ by removing the rows $\beta_0$, $\beta_i$, $\beta_j$ and the columns $\beta_{ii}$, $\beta_{ij}$, $\beta_{jk}$, and $B$ is the submatrix obtained from $M$ by removing the rows $\beta_0$, $\beta_i$, $\beta_k$ and the columns $\beta_{ii}$, $\beta_{ij}$, $\beta_{jk}$. 
\end{enumerate}
\end{lemma}

\begin{proof}
If $j$ is not in $U$, then $(\beta_j, \beta_{i_aj_a})$ entries of both $M$ and $N$ are zero for all $a \in \{1, \dots, r-2\}$. Similar to the proof of Lemma \ref{lem:ExpansionMaximalMinorA}, applying Laplace's cofactor expansion along the first three columns of $M$ and $N$ results in
    \[
    \begin{split}
    \det(M) = & \det(A)\begin{vmatrix}
        \frac12 f_{ii} & f_{ij} & f_{jk}\\
        f_i    & f_j    & 0 \\
        0      & f_i    & f_k
    \end{vmatrix} 
    + \det(B) \begin{vmatrix}
        \frac12 f_{ii} & f_{ij}   & f_{jk}\\
        0 & f_i & f_k \\
        0   & 0 & f_j
    \end{vmatrix} 
    - \det(C) \begin{vmatrix}
        f_i & f_j & 0 \\
        0   & f_i & f_k \\
        0   & 0   & f_j
    \end{vmatrix} \\
    = & \det(A)\left(\frac12 f_{ii}f_jf_k-f_{ij}f_if_k+f_{jk}f_i^2\right) + \frac12 \det(B)f_{ii}f_if_j - \det(C)f_i^2f_j
    \end{split}
    \]
    and 
    \[
    \begin{split}
        \det(N) = & \det(A)\begin{vmatrix}
        \frac12 f_{ii} & f_{ik} & \frac12 f_{jj}\\
        f_i    & f_k    & 0 \\
        0      & 0    & f_j
    \end{vmatrix} 
    + \det(B) \begin{vmatrix}
        \frac12 f_{ii} & f_{ik}   & \frac12 f_{jj}\\
        0 & 0 & f_j \\
        0   & f_i & 0
    \end{vmatrix} 
    - \det(C) \begin{vmatrix}
        f_i & f_k & 0 \\
        0   & 0 & f_j \\
        0   & f_i   & 0
    \end{vmatrix} \\
     = & \det(A) \left(\frac12 f_{ii}f_jf_k - f_{ik}f_if_j\right) - \frac12 \det(B)f_{ii}f_if_j + \det(C) f_i^2f_j.
    \end{split}
    \]
    
Taking the sum yields
\[\det(M) + \det(N) = \det(A) Q_{ij;ik}(F).\]

If $i$ is not in $U$, then the $(\beta_i,\beta_{i_aj_a})$ entries of $M$ and $N$ are zero. Therefore,
\[
\begin{split}
    \det(M) = & \det(A)\begin{vmatrix}
        \frac{1}{2}f_{ii} & f_{ij} & f_{jk}\\
        f_i    & f_j    & 0 \\
        0      & f_i    & f_k
    \end{vmatrix} 
    - \det(B) \begin{vmatrix}
        \frac12 f_{ii} & f_{ij}   & f_{jk}\\
        f_i & f_j & 0 \\
        0   & 0 & f_j
    \end{vmatrix} 
    - \det(C) \begin{vmatrix}
        f_i & f_j & 0 \\
        0   & f_i & f_k \\
        0   & 0   & f_j
    \end{vmatrix} \\
    = &  \det(A)\left(\frac12f_{ii}f_jf_k - f_{ij}f_if_k +f_{jk} f_i^2\right) - \det(B)\left(\frac{1}{2}f_{ii}f_j^2 - f_{ij}f_if_j\right) - \det(C)f_i^2f_j
\end{split}
    \]
    and 
    \[
    \begin{split}
    \det(N) = & \det(A)\begin{vmatrix}
        \frac12 f_{ii} & f_{ik} & \frac12 f_{jj}\\
        f_i    & f_k    & 0 \\
        0      & 0    & f_j
    \end{vmatrix} 
    - \det(B) \begin{vmatrix}
        \frac12 f_{ii} & f_{ik}   & \frac12 f_{jj}\\
        f_i & f_k & 0 \\
        0   & f_i & 0
    \end{vmatrix} 
    - \det(C) \begin{vmatrix}
        f_i & f_k & 0 \\
        0   & 0 & f_j \\
        0   & f_i   & 0
    \end{vmatrix} \\
        & = \det(A)\left(\frac{1}{2}f_{ii}f_jf_k - f_{ik}f_if_j\right) - \frac12 \det(B)f_{jj}f_i^2 + \det(C)f_i^2f_j.
    \end{split}
    \]
    Taking the sum yields
    \[
    \det(M) +\det(N) = Q_{ij;ik}(F)\det(A) - \frac12 Q_{ij;ij}(F)\det(B).
    \]
\end{proof}

\begin{lemma}\label{lem:ExpansionMaximalMinoriijj}
    Consider the $(r+1)$-by-$(r+1)$ matrix 
    \[M=R(:,[\beta_{ii}, \beta_{ij}, \beta_{jj}, \beta_{i_1j_1}, \dots, \beta_{i_{r-2}j_{r-2}}]).\]
    Then 
    \[
    \det(M) = 
    \frac{1}{2} Q_{ij;ij}(F)\det(A),
    \]
    where $A$ is the submatrix obtained by removing the rows $\beta_0$, $\beta_i$, $\beta_j$ and the columns $\beta_{ii}$, $\beta_{ij}$, $\beta_{jj}$. 
\end{lemma}

\begin{proof}
By the definition of $M$, the only nonzero elements in the columns $\beta_{ii}$, $\beta_{ij}$, and $\beta_{jj}$ are those corresponding to the rows $\beta_0$, $\beta_i$, and $\beta_j$. 
Consequently, $M$ is a block upper triangular matrix:
\[
M =\begin{pmatrix}
    D & \ast \\
    0 & A
\end{pmatrix},
\]
where $D = M([\beta_0, \beta_i, \beta_j], [\beta_{ii}, \beta_{ij}, \beta_{jj}])$. A straightfoward calculation shows that 
\[\det(M)=\frac{1}{2} Q_{ij;ij}(F)\det(A).\] 
\end{proof}

\begin{lemma}\label{lem:ExpansionMaximalMinorijkl_ii}
    Consider the $(r+1)$-by-$(r+1)$ matrix    
    \[M=R(:,[\beta_{ii}, \beta_{ik}, \beta_{il}, \beta_{jk}, \beta_{jl}, \beta_{i_1j_1}, \dots, \beta_{i_{r-4}j_{r-4}}]).\]
    Then 
    \[
    \det(M) = f_i^2Q_{ij;kl}(F)\det(A),
    \]
    where $A$ is the submatrix obtained by removing the rows $\beta_{0}$, $\beta_{i}$, $\beta_{j}$, $\beta_{k}$, $\beta_{l}$ and the columns $\beta_{ii}$, $\beta_{ik}$, $\beta_{il}$, $\beta_{jk}$, $\beta_{jl}$.
\end{lemma}
\begin{proof}
By the definition of $M$, the only nonzero elements in the columns $\beta_{ii}$, $\beta_{ik}$, $\beta_{il}$, $\beta_{jk}$ and $\beta_{jl}$ are in the submatrix $D = M([\beta_0, \beta_i, \beta_j, \beta_k, \beta_l], [\beta_{ii}, \beta_{ik}, \beta_{il}, \beta_{jk}, \beta_{jl}])$.
Consequently, $M$ is a block upper triangular matrix
\[
M =\begin{pmatrix}
    D & \ast \\
    0 & A
\end{pmatrix}.
\]
A straightforward calculation shows that
\[\det(M) = f_i^2 Q_{ij;ik}(F)\det(A).\]
\end{proof}

\begin{lemma}\label{lem:ExpansionMaximalMinorijkl_ij}
    Consider the $(r+1)$-by-$(r+1)$ submatrix 
    \[M=R(:,[\beta_{ij}, \beta_{ik}, \beta_{il}, \beta_{jk}, \beta_{jl}, \beta_{i_1j_1}, \dots, \beta_{i_{r-4}j_{r-4}}]).\]
    Then 
    \[
    \det(M) = 2f_if_jQ_{ij;kl}(F)\det(A),
    \]
    where $A$ is the submatrix obtained by removing the rows $\beta_{0}$, $\beta_{i}$, $\beta_{j}$, $\beta_{k}$, $\beta_{l}$ and the columns $\beta_{ij}$, $\beta_{ik}$, $\beta_{il}$, $\beta_{jk}$, $\beta_{jl}$.
\end{lemma}
\begin{proof}
By the definition of $M$, the only nonzero elements in the columns $\beta_{ij}$, $\beta_{ik}$, $\beta_{il}$, $\beta_{jk}$ and $\beta_{jl}$ are in the submatrix $D = M([\beta_0, \beta_i, \beta_j, \beta_k, \beta_l], [\beta_{ij}, \beta_{ik}, \beta_{il}, \beta_{jk}, \beta_{jl}])$.
Consequently, $M$ is a block upper triangular matrix
\[
M =\begin{pmatrix}
    D & \ast \\
    0 & A
\end{pmatrix}.
\]
A straightforward calculation shows that
\[\det(M) = 2f_if_jQ_{ij;kl}(F)\det(A).\]
\end{proof}

\begin{lemma}\label{lem:MinorReduction}
    Consider the $(r+1)$-by-$(r+1)$ matrix 
    \[M=R(:, [\beta_{ii}, \beta_{i_1j_1}, \dots, \beta_{i_{r}j_{r}}]),\] 
    where $U := \bigcup\limits_{a=1}^r\left(\{i_a\}\cup \{j_a\}\right) \subseteq S\setminus\{i\}$. 
    Then 
    \[\det(M) = - f_i\det(N),\]
    where $N$ is the submatrix of $M$ obtained by removing the $\beta_i$ row and the $\beta_{ii}$ column.
\end{lemma}
\begin{proof}
Since $i\not\in U$, the $(\beta_i, \beta_{i_aj_a})$ entries are zero for all $a\in\{ 1, \dots, r\}$. Therefore, applying Laplace's cofactor expansion along the second row $\beta_i$ of $M$ results in
\[
    \det(M) = - f_i \det(N).
\]
\end{proof}

\section{Decomposition of Second Jacobian Ideals}

With the preliminary calculations in the previous section, we are now ready to prove Conjecture \ref{conj:decomposition} (\cite[Conjecture 2.8]{Ramirez2024}) which has been proved for $n=2$ (\cite[Lemma 4.1]{Hussain2023}) and $n=3$ (\cite[Proposition 2.9]{Ramirez2024}). In this section, we confirm the conjecture for $n > 3$.

\begin{proposition}\label{prop:CalcuationMinors}
Let $F$ be a holomorphic function of $n$ variables and $M=\Jac_2(F)\left(:, [\beta_{i_1 j_1}, \dots, \beta_{i_{n+1} j_{n+1}}]\right)$ a maximal square submatrix of $\Jac_2(F)$, where $1\le i_a, j_a \le n$. Then
\[\det(M) \in \mathcal{J}_1(F)^{n-2}\mathcal{Q}(F).\]
\end{proposition}
\begin{proof}
If $\bigcup\limits_{a=1}^{n+1}\left(\{i_a\}\cup\{j_a\}\right)\neq \{1, \dots, n\}$, by Lemma \ref{lem:degenerateSubmatrix}, $\det(M)=0$. Therefore, we assume $\bigcup\limits_{a=1}^{n+1}\left(\{i_a\}\cup\{j_a\}\right) = \{1, \dots, n\}$.
    By Lemma \ref{lem:blockUpperDiagonal}, we can assume that $M$ is a block upper triangular matrix whose determinant is  
    \[\det(M)=\det(D)\prod_{i}\det(B_i),\] 
    where $B_i$ are diagonal matrices, every nonzero row index of $D$ appears in at least two column labels of $D$, and each column index of $D$ is also a row index of $D$. Suppose that the dimension of $D$ is $r+1$. By Lemma \ref{lem:doubleindices}, we know that $r\ge 2$.
    
    Since the diagonal entries of $B_i$ are the first partial derivative $f_a$ and the total rank of $B_i$ is $(n+1)-(r+1) = n-r$, the product $\prod\limits_{i}\det(B_i)$ is in $\mathcal{J}_1(F)^{n-r}$. 
    To prove the claim, it suffices to show that
    \[\det(D)\in \mathcal{J}_1(F)^{r-2}\mathcal{Q}(F).\]

The proof proceeds by induction on the dimension of $D$. 

Suppose $r + 1 = 3$ and assume that the row indices of $D$ are $0$, $i$ and $j$, where $1\le i < j\le n$. By Lemma \ref{lem:doubleindices}, the column labels of $D$ must be $\beta_{ii}$, $\beta_{ij}$ and $\beta_{jj}$. Up to row and column switching, we may assume that 
\[
D = M([\beta_0, \beta_i, \beta_j], [\beta_{ii}, \beta_{ij}, \beta_{jj}]).
\]
By Lemma \ref{lem:ExpansionMaximalMinoriijj}, 
\[\det(D)= \frac12Q_{ij;ij}(F) \in \mathcal{Q}(F).\]

Suppose $r + 1 = 4$ and assume that the row indices of $D$ are $0$, $i$, $j$, and $k$ where $1\le i < j < k\le n$. By Lemma \ref{lem:doubleindices}, we may assume that $m_i = 4$ or $m_i=m_j=3$. Then up to row and column switching, we may assume that
\[
D = M([\beta_0, \beta_i, \beta_j, \beta_k], [\beta_{ii}, \beta_{ij}, \beta_{jk}, \beta_{ik}])  \text{ or } D = M([\beta_0, \beta_i, \beta_j, \beta_k], [\beta_{ii}, \beta_{ik}, \beta_{jj}, \beta_{jk}]). 
\]

If $D = M([\beta_0, \beta_i, \beta_j, \beta_k], [\beta_{ii}, \beta_{ij}, \beta_{jk}, \beta_{ik}])$, then applying Lemma \ref{lem:ExpansionMaximalMinorB} \eqref{lem-item:mi=4} with $M=D$, along with the consequence that $N$ has duplicate columns $\beta_{ik}$, implies
\[\det(D)=f_iQ_{ij;ik}(F) \in \mathcal{J}_1(F)\mathcal{Q}(F).\]

If $D = M([\beta_0, \beta_i, \beta_j, \beta_k], [\beta_{ii}, \beta_{ik}, \beta_{jj}, \beta_{jk}])$, then applying Lemma \ref{lem:ExpansionMaximalMinorB} \eqref{lem-item:mi=3} with $N=D$, along with the consequence that $M$ has duplicate columns $\beta_{jk}$, implies
\[\det(D) = f_j Q_{ij;ik}(F) - \frac{1}{2}f_k Q_{ij;ij}(F) \in  \mathcal{J}_1(F)\mathcal{Q}(F).\]

Suppose $\det(D)\in \mathcal{J}_1(F)^{r - 2}\mathcal{Q}(F)$ for $r + 1\le d$. 
We will show that $\det(D)\in \mathcal{J}_1(F)^{r - 2}\mathcal{Q}(F)$ for $r + 1 = d + 1$.

By Lemma \ref{lem:doubleindices}, $D$ has a row index $i$ with frequency $m_i \ge 3$ and each other row index of $D$ appears in at least two column labels of $D$.

\paragraph{\textbf{Case 1:}}
Suppose that $D$ contains the column $\beta_{ii}$. Then $D$ also contains columns $\beta_{ij}$, $\beta_{jk}$. If $k=j$, then by \ref{lem:doubleindices} no index other than $i$ and $j$ has frequency $3$. Applying Lemma \ref{lem:ExpansionMaximalMinoriijj} leads to $\det(D)=\frac{1}{2}Q_{ij;ij}(F)\det(A)$. Here $A$ is a matrix of dimension $r-2$ whose entries are among the first partial derivatives of $F$. Therefore, $\det(A) \in \mathcal{J}_1(F)^{r-2}$ and consequently, $\det(D)\in  \mathcal{J}_1(F)^{r - 2}\mathcal{Q}(F)$.

Suppose that $k\ne j$. Let $N$ be the matrix obtained from $D$ by replacing the entries in the columns $\beta_{ij}$ and $\beta_{jk}$ of $D$ by the corresponding entries in the columns $\beta_{ik}$ and $\beta_{jj}$ of $\Jac_2(F)$ respectively.

If $m_i=4$, then by Lemma \ref{lem:doubleindices}, $j$ is not an index for the remaining column labels. Applying Lemma \ref{lem:ExpansionMaximalMinorB} \eqref{lem-item:mi=4} to $M=D$ leads to $\det(M) + \det(N) = Q_{ij;ik}(F)\det(A)$, where $A$ is the submatrix of $M$ obtained by removing the rows $\beta_0$, $\beta_{i}$, $\beta_j$ and the columns $\beta_{ii}$, $\beta_{ij}$, $\beta_{jk}$. Since the submatrix $A$ contains no row $\beta_0$, its entries are elements of $\mathcal{J}_1(F)$. Moreover, since the dimension of $A$ is $r + 1 -3 = r-2$, the determinant $\det(A)$ belongs to $\mathcal{J}_1(F)^{r - 2}$. 
Therefore, to show that $\det(D) \in \mathcal{J}_1(F)^{r - 2} \mathcal{Q}(F)$, it suffices to show that $\det(N)  \in \mathcal{J}_1(F)^{r - 2} \mathcal{Q}(F)$. 
From the construction of $N$, the column index $j$ appears only twice and is already in the column label $\beta_{jj}$.
Then by Lemma \ref{lem:MinorReduction}, up to a sign difference, $\det(N) = f_j\det(P)$ for an $r$-by-$r$ submatrix $P$ obtained by removing the row $\beta_j$ and the column $\beta_{jj}$. 
Moreover, each nonzero row index of $P$ appears in at least two column labels, and a column index of $P$ is also a row index of $P$ according to the definition of $D$. 
Since the dimension of $P$ is $d = r = (r - 1) + 1$, by the induction assumption, $\det(P)\in \mathcal{J}_1(F)^{(r-1) - 2}\mathcal{Q}(F)$. Consequently, $\det(N)\in \mathcal{J}_1(F)^{r-2}\mathcal{Q}(F)$ and so is $\det(D)$.

If $m_i=3$, then $i$ is not an index for the remaining column labels. 
Applying Lemma \ref{lem:ExpansionMaximalMinorB} \eqref{lem-item:mi=3} results in 
\[\det(D) \equiv \det(N) \mod \mathcal{J}_1(F)^{r-2}\mathcal{Q}(F).\] 
If $m_j = 2$, then similar to the case where $m_i=4$, we know that $\det(N)\in \mathcal{J}_1(F)^{r-2}\mathcal{Q}(F)$ by induction and hence $\det(D)\in \mathcal{J}_1(F)^{r-2}\mathcal{Q}(F)$. If $m_j=3$, then $m_k=2$. Then $D$ has another column $\beta_{kl}$, where $l\ne i$. 
If $l = j$, then $\det(D)=0$. 
Since $m_k=2$ and $D$ already has a column $\beta_{jk}$ by the assumption, the index $l$ must be different from $k$. 
Therefore, we may suppose $l\not\in\{i,j,k\}$. Then $m_l=2$ and $D$ has another column $\beta_{lm}$. Consequently, $N$ is permutation equivalent to a matrix $N'$ so that the columns $\beta_{ik}$, $\beta_{kl}$ and $\beta_{lm}$ of $N$ are the first three columns of $N'$, and the rows $\beta_0$, $\beta_i$, $\beta_l$,  $\beta_k$ of $N$ are the first four rows of $N'$. Applying Lemma \ref{lem:ExpansionMaximalMinorA} to $M=N'$ implies that 
\[\det(N') \equiv \det(N'') \mod \mathcal{J}_1(F)^{r-2}\mathcal{Q}(F),\] 
where $N''$ is the matrix corresponding to $N$ in that lemma. Note that $N''$ has the columns $\beta_{kk}$ and $m_k=2$. Therefore, by Lemma \ref{lem:MinorReduction} and the induction assumption, $\det(N'')\in \mathcal{J}_1(F)^{r-2}\mathcal{Q}(F)$. Consequently, $\det(N)$ and $\det(N')$ are in $\mathcal{J}_1(F)^{r-2}\mathcal{Q}(F)$.

\paragraph{\textbf{Case 2:}} 
Suppose that $D$ contains no column label in the form $\beta_{aa}$. By Lemma \ref{lem:doubleindices}, $D$ has a nonzero row index $i$ with $m_i\ge 3$ and at most another row index $a$ with $m_a=3$ provided that $m_i=3$. Therefore, there is a row index $j$ such that $D$ has a column $\beta_{ij}$ with $m_j=2$. We may assume that $\beta_{jk}$ is another column of $D$. Similarly, we may also assume that $m_k = 2$ and, hence, $D$ has a third column $\beta_{kl}$. Since $m_j = m_k = 2$, they do not appear in other column labels. Let $N$ be the matrix obtained from $D$ by replacing the columns $\beta_{ij}$, $\beta_{jk}$ and $\beta_{kl}$ by the columns $\beta_{il}$, $\beta_{kk}$ and $\beta_{jj}$. By Lemma \ref{lem:ExpansionMaximalMinorA}, 
$\det(D) + \det(N) = Q_{ik; jl}(F)\det(A)$. Similar to Case 1,  $\det(A) \in \mathcal{J}_1(F)^{r -2}$ and $\det(D)\in \mathcal{J}_1(F)^{r -2} \mathcal{Q}(F)$ if $\det(N) \in \mathcal{J}_1(F)^{r -2} \mathcal{Q}(F)$. Since $m_j = 2$ and $m_k = 2$, applying Lemma \ref{lem:MinorReduction} to $N$ results in $\det(N)=f_jf_k\det(P)$ for an $r - 1$ dimensional matrix $P$ with the properties that each nonzero row index appears in at least two column labels and that a column index is also a row index. 
By the induction assumption, it follows that $\det(P)\in \mathcal{J}_1(F)^{(r-2) - 2}\mathcal{Q}(F)$. Consequently, $\det(N)\in \mathcal{J}_1(F)^{r-2}\mathcal{Q}(F)$, and so is $\det(D)$.

Therefore, $\det(M)\in \mathcal{J}_1(F)^{n-2}\mathcal{Q}(F)$.
\end{proof}

\begin{proposition}\label{prop:ExistenceMinors}
    Let $F$ be a holomorphic function of $n$ variables. Then \[(\mathcal{J}_1(F))^{n-2}\mathcal{Q}(F)\subseteq\mathcal{J}_2(F).\]
\end{proposition}

\begin{proof}
Since $\mathcal{Q}(F)$ is generated by $Q_{ij;kl}(F)$ with $1\le i < j\le n$, $1\le k < l\le n$,
it suffices to show that
 $Q_{ij;kl}(F) \prod\limits_{a = 1}^{n-2}f_{j_a}$ is a combination of maximal minors of $\Jac_2(F)$ for any pairs $1\le i < j\le n$, $1\le k < l\le n$ and any sequence $1\le j_1\le \cdots\le j_{n-2} \le n$.
 
Let $U := \bigcup\limits_{a=1}^{n-2}\{j_a\}.$ Denote by $c$ the cardinality of the set $\{i\}\cup\{j\}\cup \{k\}\cup \{l\}$. Then $2\le c \le 4$.

The idea to prove the proposition is to find a matrix
\[
M = \begin{pmatrix}
    D & \ast\\
    0 & A
\end{pmatrix}
\]
such that $\det(M)=\det(D)\det(A)$,  $\det(D)= Q_{ij;kl}(F) P(F)$, and $\det(A) = \left(\prod\limits_{a = 1}^{n-2}f_{j_a}\right) / P(F)$, where $P(F)$ is a factor of $\prod\limits_{a = 1}^{n-2}f_{j_a}$.

For the idea to work, the row indices of $M$ must be all different, which we will assume.

\paragraph{\textbf{Case 1:} $c=2$.}
Since $i<j$, $k<l$ and $c=2$, we must have $(i, j) = (k, l)$. It suffices to show that $Q_{ij;ij}(F) \prod\limits_{a=1}^{n-2} f_{j_a} = \det(M)$ for an $(n+1)$-by-$(n+1)$ submatrix $M$ of $\Jac_2(F)$.

Let
\[
M = \Jac_2(F)\left([\beta_0, \beta_{i}, \beta_{j}, \beta_{i_1}, \dots, \beta_{i_{n-2}}], [\beta_{ii}, \beta_{ij},\beta_{jj}, \beta_{i_1j_1},\dots, \beta_{i_{n-2}j_{n-2}}]\right),
\] 
where $1\le i_1 <  \cdots < i_{n-2} \le n$.
Note that $M$ is a block upper triangular matrix whose diagonal blocks are $M([\beta_0, \beta_{i}, \beta_{j}], [\beta_{ii}, \beta_{ij},\beta_{jj}])$ and $M([\beta_{i_1}, \dots, \beta_{i_{n-2}}], [\beta_{i_1j_1},\dots, \beta_{i_{n-2}j_{n-2}}])$. By Lemma  \ref{lem:ExpansionMaximalMinoriijj}  and Lemma \ref{lem:detSubmatrices},  
\[
Q_{ij;ij}(F)\prod\limits_{a=1}^{n-2} f_{j_a} = 2\det(M) \in \mathcal{J}_2(F) \text{ if } c = 2.
\]

\paragraph{\textbf{Case 2:} $c=3$.} 
Since $i<j$ and $c=3$, either $i=k$, $i=l$, $j=k$ or $j=l$. Correspondingly, $Q_{ij;kl}(F) = Q_{ij;il}(F)$, $Q_{ij;kl}(F) = Q_{ij;ki}(F) = -Q_{ij;ik}(F)$, $Q_{ij;kl}(F) = Q_{ij;jl}(F) = -Q_{ji;jl}(F)$, or $Q_{ij;kl}(F) = Q_{ij;kj}(F) = Q_{ji;jk}(F)$. 
Note that up to a sign difference, $Q_{ij;kl}(F)$ is equivalent to $Q_{ab;ac}(F)$ for $a\in \{i,j\}$ and $\{b, c\} = \big(\{i\}\cup\{j\}\cup \{k\}\cup \{l\}\big) \setminus \{a\}$. 
In the following, we only show that $Q_{ij;ik}(F)\prod\limits_{a = 1}^{n-2}f_{j_a} \in \mathcal{J}_2(F)$ and other cases can be proved similarly.

\paragraph{\textbf{Subcase 2.1:} Suppose $i=j_a$ for some $1\le a \le n-2$.}

Let
\[
M = \Jac_2(F)\left([\beta_0, \beta_{i}, \beta_{j}, \beta_{k}, \beta_{i_1}, \dots, \widehat{\beta_{i_a}},\dots, \beta_{i_{n-2}}], [\beta_{ii}, \beta_{ij}, \beta_{jk}, \beta_{ik},  \beta_{i_1j_1},\dots, \widehat{\beta_{i_aj_a}}, \dots \ \beta_{i_{n-2}j_{n-2}}]\right),
\]
\[
D = M\left([\beta_0, \beta_{i}, \beta_{j}, \beta_{k}], [\beta_{ii}, \beta_{ij}, \beta_{jk}, \beta_{ik}]\right), \text{ and }
A=M([\beta_{i_1}, \dots, \widehat{\beta_{i_a}},\dots, \beta_{i_{n-2}}], [\beta_{i_1j_1},\dots, \widehat{\beta_{i_aj_a}}, \dots \ \beta_{i_{n-2}j_{n-2}}],
\]
where $1 \le i_1 < \cdots < i_{n-2} \le n$. 
Then 
\[M = \begin{pmatrix}
    D & *\\
    0 & A
\end{pmatrix},\quad
\det(M)=\det(D)\det(A),\quad \text{ and }\quad
\det(A) = \left(\prod_{a=1}^{n-2}f_{j_a}\right)/(f_i).\]
Applying Lemma \ref{lem:ExpansionMaximalMinorB} \eqref{lem-item:mi=4} to $M=D$, along with the consequence that $N$ has duplicate columns $\beta_{ik}$, implies that 
\[\det(D)= f_iQ_{ij;ik}(F) - \det(N) = f_iQ_{ij;ik}(F).\]
Therefore,
\[
Q_{ij;ik}(F)\prod_{a=1}^{n-2}f_{j_a}= \det(D)\det(A) = \det(M)\in\mathcal{J}_2(F) ~\text{if}~ i\in U.
\]

\paragraph{\textbf{Subcase 2.2:} $i\not\in U$.} 
Let
\[
M = \Jac_2(F)\left([\beta_0, \beta_{i}, \beta_{j}, \beta_{i_1}, \dots, \beta_{i_{n-2}}], [\beta_{ii}, \beta_{ij},\beta_{jk}, \beta_{i_1j_1},\dots, \beta_{i_{n-2}j_{n-2}}]\right),
\] 
and
\[
N = 
\Jac_2(F)\left([\beta_0, \beta_{i}, \beta_{j}, \beta_{i_1}, \dots, \beta_{i_{n-2}}], [\beta_{ii}, \beta_{ik},\beta_{jj}, \beta_{i_1j_1}, \dots, \beta_{i_{n-2}j_{n-2}}]\right),
\]
where $1 \le i_1 < \cdots < i_{n-2} \le n$.
Let $A=M([\beta_{i_1}, \dots, \beta_{i_{n-2}}], [\beta_{i_1 j_1},\dots, \beta_{i_{n-2}j_{n-2}}])$.
By Lemma \ref{lem:detSubmatrices}, we have
\[
\det(A) = \prod_{a=1}^{n-2}f_{j_a}. 
\]
Note that $i\not\in \{i_1,\cdots, i_{n-2}\}\bigcup U$.
Then by Lemma \ref{lem:ExpansionMaximalMinorB} \eqref{lem-item:mi=3}, it follows that 
\[Q_{ij;ik}(F) \prod_{a=1}^{n-2}f_{j_a} = Q_{ij;ik}(F)\det(A) = \det(M) + \det(N) + \frac{1}{2} Q_{ij;ij}(F)\det(B).\]
Since the nonzero entries of $B$ are in $\mathcal{J}_1(F)$ and $B$ is of dimension $n-2$, $\det(B)\in\mathcal{J}_1(F)^{n-2}$. By Case 1, 
\[Q_{ij;ij}(F)\det(B) \in Q_{ij;ij}(F)\mathcal{J}_1(F)^{n-2}\subseteq \mathcal{J}_2(F).\]
Therefore, 
\[Q_{ij;ik}(F) \prod_{a=1}^{n-2}f_{j_a} \in \mathcal{J}_2(F) ~\text{if}~ i\not\in U.\]

It then follows from Subcase 2.1 and Subcase 2.2 that
\[
Q_{ij;ik}(F) \prod_{a=1}^{n-2}f_{j_a} \in \mathcal{J}_2(F)
~\text{if}~ c=3.
\]
\paragraph{\textbf{Case 3:} $c=4$.} In this case, we focus on constructing $M$ and $N$ such that $\det(M)$ or $\det(M)\pm \det(N)$ is in the form $Q_{ij;kl}(F)P(F)\det(A) = Q_{ij;kl}(F)\prod\limits_{a=1}^{n-2}f_{j_a}$, where $P(F)$ is a factor of $\prod\limits_{a=1}^{n-2}f_{j_a}$, and $A$ is a submatrix of $M$.

\paragraph{\textbf{Subcase 3.1:} $t:=j_a=j_b\in \{i,j,k,l\}$.} 

If $t \in \{i, j\}$, then we let
\[
\begin{aligned}
M = & \Jac_2(F)\left([\beta_0, \beta_{i}, \beta_{j}, \beta_{k}, \beta_{l}, \beta_{i_1}, \dots, \widehat{\beta_{i_a}}, \dots, \widehat{\beta_{i_b}}, \dots, \beta_{i_{n-2}}],\right. \\
&\hspace{1in}\left.[\beta_{tt}, \beta_{ik},\beta_{il},\beta_{jk}, \beta_{jl}, \beta_{i_1j_1},\dots, \widehat{\beta_{i_a j_a}},\dots, \widehat{\beta_{i_b j_b}},\dots, \beta_{i_{n-2}j_{n-2}}]\right),
\end{aligned}
\]
where $1 \le i_1 < \cdots < i_{n-2} \le n$. 

If $t \in \{k, l\}$, then we may choose columns with $i$ switched with $k$ and $j$ switched with $l$ in the above defined $M$. More precisely, we let
\[
\begin{aligned}
M = & \Jac_2(F)\left([\beta_0, \beta_{k}, \beta_{l}, \beta_{i}, \beta_{j}, \beta_{i_1}, \dots, \widehat{\beta_{i_a}}, \dots, \widehat{\beta_{i_b}}, \dots, \beta_{i_{n-2}}],\right. \\
&\hspace{1in}\left.[\beta_{ki},\beta_{kj},\beta_{li}, \beta_{lj}, \beta_{tt}, \beta_{i_1j_1},\dots, \widehat{\beta_{i_a j_a}},\dots, \widehat{\beta_{i_b j_b}},\dots, \beta_{i_{n-2}j_{n-2}}]\right),
\end{aligned}
\]
where $1 \le i_1 < \cdots < i_{n-2} \le n$.

Let
\[
A=M([\beta_{i_1}, \dots, \widehat{\beta_{i_a}}, \dots, \widehat{\beta_{i_b}}, \dots, \beta_{i_{n-2}}], [\beta_{i_1j_1},\dots, \widehat{\beta_{i_a j_a}},\dots, \widehat{\beta_{i_b j_b}},\dots, \beta_{i_{n-2}j_{n-2}}]).
\]

Applying Lemma \ref{lem:detSubmatrices} to $A$ results in
\[
\det(A) = \left(\prod_{a=1}^{n-2}f_{j_a}\right)/f_t^2. 
\]

Because $Q_{ij;kl}(F) = -Q_{ji;kl}(F) = Q_{kl;ij}(F) = -Q_{lk;ij}(F)$, applying Lemma \ref{lem:ExpansionMaximalMinorijkl_ii} to $M$---directly when $t \in \{i, k\}$, or with necessary row and column switching when $t\in\{j,l\}$---leads to
\[\det(M) = f_t^2 Q_{ij;kl}(F) \det(A) = Q_{ij;kl}(F) \prod_{a=1}^{n-2}f_{j_a}. 
\]
Therefore, 
\[
Q_{ij;kl}(F) \prod_{a=1}^{n-2}f_{j_a} \in \mathcal{J}_2(F)
\text{ if there exist } j_a=j_b\in \{i,j,k,l\}.
\]

\paragraph{\textbf{Subcase 3.2:} $\{j_a, j_b\} \in \{\{i, j\}, \{k, l\}\}$.}

Because $Q_{ij;kl}(F) = -Q_{ji;kl}(F) = -Q_{ij;lk}(F) = Q_{kl;ij}(F)$, switching $j_a$ with $j_b$ will only result in a sign change. Therefore, we may assume that $(j_a, j_b)=(i,j)$ or $(j_a, j_b) = (k,l)$.

If $(j_a, j_b)=(i,j)$, we let
\[
\begin{aligned}
M = & \Jac_2(F)\left([\beta_0, \beta_{i}, \beta_{j}, \beta_{k}, \beta_{l}, \beta_{i_1}, \dots, \widehat{\beta_{i_a}}, \dots, \widehat{\beta_{i_b}}, \dots, \beta_{i_{n-2}}],\right. \\
&\hspace{1in}\left.[\beta_{ij}, \beta_{ik}, \beta_{il},\beta_{jk}, \beta_{jl}, \beta_{i_1j_1},\dots, \widehat{\beta_{i_a j_a}},\dots, \widehat{\beta_{i_b j_b}},\dots, \beta_{i_{n-2}j_{n-2}}]\right).
\end{aligned}
\]

If $(j_a,j_b)=(k, l)$, 
we let
\[
\begin{aligned}
M = & \Jac_2(F)\left([\beta_0, \beta_{k}, \beta_{l}, \beta_{i}, \beta_{j}, \beta_{i_1}, \dots, \widehat{\beta_{i_a}}, \dots, \widehat{\beta_{i_b}}, \dots, \beta_{i_{n-2}}],\right. \\
&\hspace{1in}\left.[\beta_{kl}, \beta_{ki}, \beta_{kj},\beta_{li}, \beta_{lj}, \beta_{i_1j_1},\dots, \widehat{\beta_{i_a j_a}},\dots, \widehat{\beta_{i_b j_b}},\dots, \beta_{i_{n-2}j_{n-2}}]\right).
\end{aligned}
\]

Let 
\[
A=M([\beta_{i_1}, \dots, \widehat{\beta_{i_a}}, \dots, \widehat{\beta_{i_b}}, \dots, \beta_{i_{n-2}}], [\beta_{i_1j_1},\dots, \widehat{\beta_{i_a j_a}},\dots, \widehat{\beta_{i_b j_b}},\dots, \beta_{i_{n-2}j_{n-2}}]),
\]
and $D$ the complement submatrix of $A$ in $M$.

By Lemma \ref{lem:detSubmatrices}, we have
\[
\det(A) = \left(\prod_{a=1}^{n-2}\right)/(f_{j_a}f_{j_b}).
\]
Applying Lemma \ref{lem:ExpansionMaximalMinorijkl_ij} to $M$ leads to
\[\frac{1}{2}\det(M) = \frac{1}{2}\det(D)\det(A) = f_{j_a}f_{j_b} Q_{ij;kl}(F)\left(\prod_{a=1}^{n-2}\right)/(f_{j_a}f_{j_b}) = Q_{ij;kl}(F)\prod_{a=1}^{n-2}f_{j_a}.\]
Therefore, 
\[
Q_{ij;kl}(F)\prod_{a=1}^{n-2}f_{j_a} \in \mathcal{J}_2(F)
\text{ if } \{j_a, j_b\} \in \{\{i, j\}, \{k, l\}\}.
\]

\paragraph{\textbf{Subcase 3.3:} $\{j_a, j_b\} \cap \{i, j, k, l\} \neq\emptyset$ but $\{j_a,j_b\}\not\in\{\{i,j\}, \{k,l\}\}$.} 

Because switching $i$ with $j$ or $k$ with $l$ in $Q_{ij;kl}(F)$ only results in a sign change, we may assume that $j_a\in \{i,k\}$, otherwise, we take $j_b$ as $j_a$.

Since $(j_a, j_b)\in \{(i,j), (k,l)\}$ has been considered in Subcase 3.2, if $j_a=i$, then we can assume that $j_b\neq j$, and $\{k, l\} \not\subset U$, that is, $k\not\in U$ or $l\not\in U$. Similarly, if $j_a = k$, then we may assume that $j_b\neq l$ and $i \not\in U$ or $j\not\in U$.

Suppose that $j_a = i$, $j\not\in U$, and $k\not\in U$. Let
\[
\begin{aligned}
M = & \Jac_2(F)\left([\beta_0, \beta_{i}, \beta_{j}, \beta_{k}, \beta_l, \beta_{i_1}, \dots, \widehat{\beta_{i_a}},\dots, \widehat{\beta_{i_b}}, \dots, \beta_{i_{n-2}}]\right.,\\
& \hspace{1in}\left.[\beta_{ik}, \beta_{jk}, \beta_{jl}, \beta_{ii},  \beta_{i_1j_1},\dots, \widehat{\beta_{i_a j_a}},\dots, \beta_{i_{n-2}j_{n-2}}]\right),
\end{aligned}
\]
\[
\begin{aligned}
N = & \Jac_2(F)\left([\beta_0, \beta_{i}, \beta_{j}, \beta_{k}, \beta_l, \beta_{i_1}, \dots, \widehat{\beta_{i_a}}, \cdots, \widehat{\beta_{i_b}}, \dots, \beta_{i_{n-2}}]\right.,\\
& \hspace{1in}\left.[\beta_{il}, \beta_{jj}, \beta_{kk}, \beta_{ii},  \beta_{i_1j_1},\dots, \widehat{\beta_{i_a j_a}},\dots, \beta_{i_{n-2}j_{n-2}}]\right),
\end{aligned}
\]
\[A=M([\beta_i, \beta_l,  \beta_{i_1}, \dots, \widehat{\beta_{i_a}}, \dots, \widehat{\beta_{i_b}}, \dots, \beta_{i_{n-2}}],[\beta_{ii},\beta_{i_1j_1},\dots, \widehat{\beta_{i_a j_a}},\dots, \beta_{i_{n-2}j_{n-2}}]),\]
and $B=A([\beta_l, \beta_{i_1}, \dots, \widehat{\beta_{i_a}}, \dots, \widehat{\beta_{i_b}},\dots,\beta_{i_{n-2}}],[\beta_{i_1j_1}, \dots, \widehat{\beta_{i_a j_a}},\dots,  \beta_{i_{n-2}j_{n-2}}])$,
where $1 \le i_1 < \cdots < i_{n-2} \le n$ and $\{i, l\} = \{i_a, i_b\}$. Without loss of generality, in the matrices defined above, we may assume that $i_a = i$. If instead $i_b = i$, after switching the label $a$ with the label $b$ in the above setting, the subsequent arguments remain valid under this relabeling. 

Notice that $j, k\not\in U\cup\{i_1, \dots, i_{n-2}\}$.
Applying Lemma \ref{lem:ExpansionMaximalMinorA} to $M$ and $N$ results in
\[
\det(M)+\det(N)= -Q_{ij;kl}(F)\det(A).
\]
Applying Lemma \ref{lem:detSubmatrices} to $B$ leads to 
\[\det(A)=f_i\cdot (-1)^{b - 2} \det(B) = (-1)^{b} f_i\left(\prod_{a=1}^{n-2}f_{j_a}\right)/f_{j_a} =  (-1)^{b} \prod_{a=1}^{n-2}f_{j_a}.\]
Therefore,
\[
Q_{ij;kl}(F)\prod_{a=1}^{n-2}f_{j_a} = (-1)^{b+1}(\det(M)+\det(N)) \in \mathcal{J}_2(F).
\]

Suppose $j_a = i$, $j\not\in U$ and $l\not\in U$. Switching $k$ with $l$ in the above defined matrices and applying the same arguments leads to 
\[
Q_{ij;lk}(F)\prod_{a=1}^{n-2}f_{j_a} \in \mathcal{J}_2(F).
\]

If $j_a = k$, then applying the same arguments to $M$ and $N$ obtained from the original $M$ and $N$ by switching $i$ with $k$ and $j$ with $l$ leads to
\[
Q_{kl;ij}(F)\prod_{a=1}^{n-2}f_{j_a} \in \mathcal{J}_2(F) \quad \text{ or } \quad Q_{kl;ji}(F)\prod_{a=1}^{n-2}f_{j_a} \in \mathcal{J}_2(F).
\]

Therefore, 
\[
Q_{ij;kl}(F)\prod_{a=1}^{n-2}f_{j_a} \in \mathcal{J}_2(F) \text{ if } \{j_a, j_b\} \cap \{i, j, k, l\} \neq\emptyset \text{ and } \{j_a,j_b\}\not\in\{\{i,j\}, \{k,l\}\}.
\]

Up to now, through the subcases 3.1-3.3, we have proved that
\[
Q_{ij;kl}(F)\prod_{a=1}^{n-2}f_{j_a} \in \mathcal{J}_2(F) \text{~if~} \{i,j,k,l\} \cap \bigcup_{a=1}^{n-2}\{j_a\} \ne \emptyset.
\]
\paragraph{\textbf{Subcase 3.4:} $\{i, j, k, l\}\cap U =\emptyset$.} 

Let
\[
M =  \Jac_2(F)\left([\beta_0, \beta_{i}, \beta_{j}, \beta_{k}, \beta_{l}, \beta_{i_1}, \dots, \widehat{\beta_{i_a}}, \dots, \widehat{\beta_{i_b}}, \dots, \beta_{i_{n-2}}],[\beta_{ik}, \beta_{jk}, \beta_{jl}, \beta_{i_1j_1},\dots, \beta_{i_{n-2}j_{n-2}}]\right),
\]
\[
N =  \Jac_2(F)\left([\beta_0, \beta_{i}, \beta_{j}, \beta_{k}, \beta_{l}, \beta_{i_1}, \dots, \widehat{\beta_{i_a}}, \dots, \widehat{\beta_{i_b}}, \dots, \beta_{i_{n-2}}],
[\beta_{il}, \beta_{jj}, \beta_{kk}, \beta_{i_1j_1},\dots, \beta_{i_{n-2}j_{n-2}}]\right),
\]
and $A=M\left([\beta_i, \beta_l, \beta_{i_1}, \dots, \widehat{\beta_{i_a}}, \dots, \widehat{\beta_{i_b}}, \dots, \beta_{i_{n-2}}], [\beta_{i_1j_1},\dots, \beta_{i_{n-2}j_{n-2}}]\right)$,
where $1\le i_1\le\cdots\le i_{n-2}\le n$, $i_a = i$ and $i_b =l$. Similar to Subcase 3.3, if $i_a=l$ and $i_b=i$, after switching the label $a$ with the label $b$ in the above setting, the subsequent arguments remain valid under this relabeling.

By Lemma \ref{lem:detSubmatrices}, we have
\[
\det(A) = (-1)^{(a+b-1)}\prod_{a=1}^{n-2}f_{j_a}.
\]
Since $j, k\not\in U \cup \{i_1,\dots, i_{n-2}\}$, by Lemma \ref{lem:ExpansionMaximalMinorA},
\[
\det(M)+\det(N) = - Q_{ij;kl}(F) \det(A) = (-1)^{a+b} Q_{ij;kl}(F) \prod_{a=1}^{n-2}f_{j_a}
\]
which implies
\[
Q_{ij;kl}(F)\prod_{a=1}^{n-2}f_{j_a} \in \mathcal{J}_2(F) \text{ if }  \{i, j, k, l\} \cap U = \emptyset.
\]

From Case 1-3, we see that
\[Q_{ij;kl}(F)\prod_{a=1}^{n-2}f_{j_a}\in \mathcal{J}_2(F) ~\text{for}~ 1\le i < j\le n, 1\le k < l\le n, \text{ and } 1\le j_1 \le \cdots \le j_{n-2} \le n.\]
Therefore,
\[
\mathcal{J}_1(F)^{n-2}\mathcal{Q}(F)\subseteq \mathcal{J}_2(F).
\]

\end{proof}

Lemma \ref{lem:detSubmatrices}, Proposition \ref{prop:CalcuationMinors}, and Proposition \ref{prop:ExistenceMinors} together lead to the following theorem which has been known for $n=2$ \cite{Hussain2023} and $n=3$ \cite{Ramirez2024}.

\begin{theorem}\label{thm:main}
    Let $F$ be a holomorphic function of $n$ variables. The second Jacobian ideal $\mathcal{J}_2(F)$ admits the following decomposition:
    \[
    \mathcal{J}_2(F) = \mathcal{J}_1(F)^{n+1} + \mathcal{J}_1(F)^{n-2}\mathcal{Q}(F).
    \]
\end{theorem}

\begin{proof}
Let $M$ be an $(n+1)$-by-$(n+1)$ non-degenerate submatrix of $\Jac_2(F)$. 

Note that the entries $(\beta_i,\beta_{0k})$ are zero if $i\neq 0$. If $M$ contains a column $\beta_{0k}$, then $\det(M)=f_k M_{\beta_0,\beta_{0k}}$, where $M_{\beta_0,\beta_{0k}}$ is the cofactor of $M(\beta_0, \beta_{0k})$ and the entries of the associated complement submatrix of $M(\beta_0,\beta_{0k})$ are 0 or $f_a$, where $1\le a\le n$. Therefore, $\det(M)$ is a degree $n+1$ homogeneous polynomial in $f_1$, $\dots$, $f_n$, which implies $\det(M)\in \mathcal{J}_1(F)^{n+1}$. 
Assume that $M$ contains only the columns $\beta_{ab}$, where $1 \le a, b \le n$. 
By Proposition \ref{prop:CalcuationMinors}, the determinant $\det(M)$ is in $\mathcal{J}_1(F)^{n-2}\mathcal{Q}(F)$.
Therefore, 
\[\mathcal{J}_2(F)\subseteq \mathcal{J}_1(F)^{n+1} + \mathcal{J}_1(F)^{n-2}\mathcal{Q}(F).\]

Conversely, for a monomial $\prod\limits_{a=1}^{n+1}f_{j_a}$ with $1\le j_1\le\cdots\le j_{n+1}\le n$, let
\[M=\Jac_2(F)\left(:,[\beta_{0j_1}, \beta_{1j_2}, \dots, \beta_{nj_{n+1}}]\right).\]
By Lemma \ref{lem:detSubmatrices}, it follows that 
\[\prod_{a=1}^{n+1}f_{j_a} = \det(M) \in \mathcal{J}_2(F).\]
Moreover, by Proposition \ref{prop:ExistenceMinors}, $\mathcal{J}_1(F)^{n-2}\mathcal{Q}(F) \subset \mathcal{J}_2(F)$. 
Therefore,
\[
\mathcal{J}_1(F)^{n+1} + \mathcal{J}_1(F)^{n-2}\mathcal{Q}(F)\subseteq \mathcal{J}_2(F).
\]

Consequently,
\[\mathcal{J}_2(F) = \mathcal{J}_1(F)^{n+1} + \mathcal{J}_1(F)^{n-2}\mathcal{Q}(F).\]

\end{proof}

\begin{example}
    Let $F= x^3-y^2\in \mathbb{C}[x, y]$. Then $f_1 = 3x^2$, $f_2 = -2y$, $f_{11} = 6x$, $f_{12}=f_{21}=0$, and $f_{22}=-2$. By Theorem \ref{thm:main}
    \[\mathcal{J}_2(F) = (3x^2, -2y)^3 + (6x (-2y)^2 + (-2)(3x^2)^2).\]
    Compared with \cite[Example 2.1.4]{Le2024}, we notice that \[(3x^2, -2y)^3 + (6x (-2y)^2 +  (-2)(3x^2)^2) = (x^6, x^4y, x^2y^2, y^3, 4xy^2-3x^4).\]
\end{example}

In \cite[Proposition 2.18]{Le2024} it is shown that the $k$ -th Jacobian ideal $\mathcal{J}_k(F)$ is contained in $\mathcal{J}_1(F)^{\binom{n-2+k}{n-1}}$. For $k=2$, this can also be confirmed by Theorem \ref{thm:main}. Note that $\mathcal{Q}(F)\subset \mathcal{J}_1(F)^2\mathcal{S}$, where $\mathcal{S}=(f_{ij})_{1\le i, j\le n}$ is the ideal generated by second-order partial derivatives. It follows that
\[\mathcal{J}_2(F) = \mathcal{J}_1(F)^{n+1} + \mathcal{J}_1(F)^{n-2}\mathcal{Q}(F)\subset \mathcal{J}_1(F)^{n+1} + \mathcal{J}_1(F)^{n}\mathcal{S} = \mathcal{J}_1(F)^n\big(\mathcal{J}_1(F) + \mathcal{S}\big) \subset \mathcal{J}_1(F)^n.\]

Theorem \ref{thm:main} and \cite[Proposition 2.18]{Le2024} suggest that $\mathcal{J}_k(F)$ is closely related to $\mathcal{J}_1(F)$.
\begin{question}
    For $k\ge 3$, does the $k$-th Jacobian ideal $\mathcal{J}_k(F)$ have a decomposition structure similar to $\mathcal{J}_2(F)$?
\end{question}

\section{Higher Nash blow-up local algebras}

In this section, we present an application of Theorem \ref{thm:main} to \cite[Conjecture 1.5]{Hussain2023}. The conjecture in the case that $n=2$ and $k=2$ was verified in the same paper. The conjecture for $n=3$ and $k=2$ was confirmed in \cite{Ramirez2024}. Recently, in \cite{Le2024}, the conjecture's general validity was established using properties of Fitting ideals. In this section, we provide an elementary proof of \cite[Theorem 2.24.]{Le2024} for $k=2$ using Theorem \ref{thm:main}.

We denote by $\mathcal{O}_n=\mathbb{C}\{z_1, \dots, z_n\}$ the algebra of germs of holomorphic functions at the origin of $\mathbb{C}^n$.

In \cite{Hussain2023}, the authors defined higher Nash blow-up local algebras for an isolated hypersurface singularity as follows.

\begin{definition}[{\cite[Definition 1.3]{Hussain2023}}]
    Let $F$ be a holomorphic function of $n$ variables with an isolated singularity at $0\in\mathbb{C}^n$. The \textbf{$k$-th Nash blow-up local algebra} of the hypersurface $V=V(F)$ is defined to be:
    \[
    \mathcal{M}_k(V):= {\mathcal{O}_n}/{(F, \mathcal{J}_k(F))},
    \]
    where $\mathcal{J}_k(F)$ is the $k$-th Jacobian ideal.
\end{definition}

\begin{theorem}\label{thm:conjecturen=3}
    The second Nash blow-up algebra of a hypersurface with an isolated singularity is a contact invariant.
\end{theorem}
\begin{proof}
Let $F$ be a hypersurface in $\mathbb{C}^n$ with an isolated singularity at $0$, $u$ a unit in $\mathcal{O}_n$, and $\phi: (\mathbb{C}^n, 0) \to (\mathbb{C}^n, 0)$ an isomorphic germ. Let $G:=u\cdot(F\circ\phi)$. Since $u$ is a unit, as ideals, $(G)=(F\circ\phi)$. 
By \cite[Corollary 2.20]{Le2024},
\[
\mathcal{J}_2(G) = \mathcal{J}_2(F\circ \phi).
\] 
Consequently,
\[
\mathcal{O}_n/(G, \mathcal{J}_2(G)) = \mathcal{O}_n/(F\circ\phi, \mathcal{J}_2(F\circ \phi)).
\]

By the chain rule of Jacobian (see \cite[Lemma 3.3]{Ramirez2024}), it follows that 
\[
\mathcal{J}_1(F\circ \phi)\subseteq \phi^*\mathcal{J}_1(F).
\]

By the characterization of $Q_{ij;kl}(F)$ using Kronecker product (\cite[Lemma 2.6]{Ramirez2024}) and the chain rule of Hessian (\cite[Lemma 3.5]{Ramirez2024}), it follows that, for any $1\le i,j,k,l\le n$,
\[
\begin{aligned}
Q_{ij;kl}(F \circ \phi) = & \Vectorization\Big(\Hess(F\circ \phi)\Big)^T \cdot E_{ij}\otimes E_{kl} \cdot \Vectorization\Big(\Jac(F\circ\phi)^T\Jac(F\circ \phi)\Big)\\
= & \Vectorization\Big(\Jac(\phi)^T \phi^*\Hess(F)\Jac(\phi)\Big)^T \cdot E_{ij}\otimes E_{kl} \cdot \Vectorization\Big(\Jac(\phi)^T \phi^*\Big(\Jac(F)^T\Jac(F)\Big) \Jac(\phi)\Big)\\
& + \Vectorization\Big(\sum_{k=1}^n(\phi^*f_k) \Hess(\phi_k)\Big)^T \cdot E_{ij}\otimes E_{kl} \cdot \Vectorization\Big(\Jac(\phi)^T \phi^*\Big(\Jac(F)^T\Jac(F)\Big) \Jac(\phi)\Big),
\end{aligned}
\]
where the first term is in $\phi^*\mathcal{Q}$ and the second terms is in $\phi^*(\mathcal{J}_1(F)^3)$. Therefore,
\[
\mathcal{Q}(F\circ\phi) \subseteq \phi^*\big(\mathcal{Q}(F) + \mathcal{J}_1(F)^3\big).
\]
By Theorem \ref{thm:main}, it follows that
\[\mathcal{J}_2(F\circ\phi)=\mathcal{J}_1(F\circ\phi)^{n+1} + \mathcal{J}_1(F\circ\phi)^{n-2}\mathcal{Q}(F\circ\phi)\subseteq \phi^*(\mathcal{J}_1(F)^{n+1} + \mathcal{J}_1(F)^{n-2}\mathcal{Q}(F)) = \phi^*\mathcal{J}_2(F).
\]

Let $\varphi=\phi^{-1}$ and $v=u^{-1}$. Then $F=v\cdot(G\circ\varphi)$ which implies
$(F) = (G\circ\varphi)$ as ideals. Similarly, we get
\[\mathcal{J}_2(F) = \mathcal{J}_2(G\circ\varphi)\]
and 
\[
\mathcal{J}_2(G\circ\varphi)\subseteq \varphi^*\mathcal{J}_2(G).
\]
It follows that
\[
(F\circ\phi, \mathcal{J}_2(F\circ\phi)) \subseteq \phi^*(F, \mathcal{J}_2(F)) = \phi^*(G\circ\varphi, \mathcal{J}_2(G\circ\varphi)) \subseteq \phi^*\varphi^*(G, \mathcal{J}_2(G)) = (G, \mathcal{J}_2(G)) = (F\circ\phi, \mathcal{J}_2(F\circ\phi)).
\]
Consequently,
\[
(F\circ\phi, \mathcal{J}_2(F\circ\phi)) = \phi^*(F, \mathcal{J}_2(F)).
\]
Since $\phi$ is an isomorphism, then $\phi^*: \mathcal{O}_n/(F, \mathcal{J}_2(F)) \to \mathcal{O}_n/\phi^*(F, \mathcal{J}_2(F))$ is an isomorphism.

We then have an isomorphism
\[ \mathcal{O}_n/(F, \mathcal{J}_2(F)) \overset{\sim}{\to} \mathcal{O}_n/\phi^*(F, \mathcal{J}_2(F)) = \mathcal{O}_n/(F\circ\phi, \mathcal{J}_2(F\circ \phi)) = \mathcal{O}_n/(G, \mathcal{J}_2(G)).\]
\end{proof}


\begin{thebibliography}{HMYZ23}

\bibitem[BD20]{Barajas2020}
Paul Barajas and Daniel Duarte.
\newblock On the module of differentials of order n of hypersurfaces.
\newblock {\em Journal of Pure and Applied Algebra}, 224(2):536--550, February
  2020.

\bibitem[DS24]{Dreau2024}
Yann~Le Dréau and Julien Sebag.
\newblock Arc scheme and higher differential forms.
\newblock {\em Osaka Journal of Mathematics}, 61(3):381--390, July 2024.
\newblock Publisher: Osaka University and Osaka Metropolitan University,
  Departments of Mathematics.

\bibitem[Dua17]{Duarte2017}
Daniel Duarte.
\newblock Computational aspects of the higher {{Nash}} blowup of hypersurfaces.
\newblock {\em Journal of Algebra}, 477:211--230, May 2017.

\bibitem[HMYZ23]{Hussain2023}
Naveed Hussain, Guorui Ma, Stephen S.~T. Yau, and Huaiqing Zuo.
\newblock Higher {{Nash}} blow-up local algebras of singularities and its
  derivation {{Lie}} algebras.
\newblock {\em Journal of Algebra}, 618:165--194, March 2023.

\bibitem[LY24]{Le2024}
Quy~Thuong Lê and Takehiko Yasuda.
\newblock Higher {Jacobian} ideals, contact equivalence and motivic zeta
  functions.
\newblock {\em Revista Matemática Iberoamericana}, October 2024.

\bibitem[RSY24]{Ramirez2024}
Tabitha Ramirez, Hao Shen, and Fei Ye.
\newblock {A computational approach to higher Jacobian ideals}.
\newblock {\em Involve, a Journal of Mathematics}, 2024.
\newblock To appear.

\bibitem[Yas07]{Yasuda2007}
Takehiko Yasuda.
\newblock Higher {Nash} blowups.
\newblock {\em Compositio Mathematica}, 143(6):1493--1510, November 2007.
\newblock tex.ids= Yasudaa.

\end{thebibliography}
\end{document}